\documentclass[a4paper,draft]{amsproc}
\usepackage{amssymb}
\usepackage[hyphens]{url} 
\theoremstyle{plain}
 \newtheorem{thm}{Theorem}[section]
 \newtheorem{prop}{Proposition}[section]
 \newtheorem{lem}{Lemma}[section]
 \newtheorem{cor}{Corollary}[section]
\theoremstyle{definition}
 
 \newtheorem{dfn}{Definition}[section]
\theoremstyle{remark}
 \newtheorem{rem}{Remark}[section]
 \numberwithin{equation}{subsection}

\renewcommand{\leq}{\leqslant}
\renewcommand{\geq}{\geqslant}
\renewcommand{\setminus}{\smallsetminus}
\title[Description of the set of strictly regular QBO and examples]{Description of the set of strictly regular quadratic bistochastic operators and examples}

\subjclass[2010]{}


\author[M.Makhmudov]{Mirmukhsin Makhmudov} %% Please write ful names, avoid initials

\address{M.Makhmudov, Faculty of Mathematics,National University of Uzbekistan, Tashkent, Uzbekistan}

\email{mirmukhsinmath@gmail.com}

\begin{document}

{\begin{flushleft}\baselineskip9pt\scriptsize
%PUBLICATIONS DE L'INSTITUT MATH\'EMATIQUE\newline
%Nouvelle s\'erie, tome ??(1??)) (201?), od--do \hfill DOI: \\
%%MANUSCRIPT
\end{flushleft}}
\vspace{18mm} \setcounter{page}{1} \thispagestyle{empty}

\begin{abstract}
 The present paper focuses on the dynamical systems
of the quadratic bistochastic operators (QBO) on the standard simplex. In the paper,
we show the character of connection of the dynamical systems of a bistochastic operator with
 the dynamical systems of the extreme bistochastic operators. In addition,
 we prove that almost all quadratic bistochastic operators are strictly regular and give a description of the
strictly regular quadratic bistochastic operators in the convex polytope QBO. Furthermore, the
density of the set of strictly regular QBO in the set of QBO is proved
and nontrivial examples of strictly regular bistochastic operators are given.
\end{abstract}

\maketitle

\section{Introduction}  %% Please avoid complex formulas in (sub)titles

Many processes in population biology can be associated with some nonlinear dynamical systems. In particular, dynamical systems generated by quadratic stochastic operators (QSOs) formulate many problems of population genetics. In general, dynamical systems of quadratic stochastic operators are very complex and difficult. Therefore, dynamical systems of a certain type of QSO are investigated. Quadratic bistochastic operators (QBO) are one type of QSO. An interesting property of dynamical systems of QBO is that the trajectory of any initial point converges to a periodic orbit. In other words, the  $\omega-$limit set of any starting point is always finite.

This article appeared at the intersection of several branches of mathematics, such as the theory of convex polytopes, the theory of majorization, and the theory of QSO. In the article, we give an algebraic expression of the relative interior points of convex polytope and prove the theorem on periodic points of bistochastic operators. In addition, we prove the strict regularity of all operators in the relative interior of the polytope QBO and give non-trivial examples of strictly regular quadratic bistochastic operators.

\section{Preliminaries}  %% Please avoid complex formulas in (sub)titles

In this section, we give some important definitions in the theory of convex polytope, the theory of majorization, and the theory of QSO. Therefore, we recall some concepts in affine geometry and the theory of dynamical systems. First, we define an affine structure on $R^d$. Throughout the article, we do not distinguish the concept of a point from a vector, and this does not lead to confusion.

The affine (convex) combination of $a_1, a_2,...,a_s\in R^d$ is, by definition, $\lambda_1 a_1+\lambda_2 a_2+...+\lambda_s a_s$ where $\sum_{j=1}^s \lambda_j=1$ and $\lambda_j\in R$ ($\lambda_j\in R_+$) for $j=\overline{1,s}$. A nonempty subset, $L\subset R^d$, is called an affine subspace of $R^d$ if it is closed with respect to affine combinations of its elements. It is clear that a nonempty intersection of affine spaces is also an affine space, hence the affine hull, $Aff(M)$, of a set, $M$, in $R^d$ is defined by the intersection of all affine subspaces including  $M$. One can easily prove that
\begin{equation}\label{Eq0}
Aff(M)=\{\sum_{i=1}^t\mu_i z_i:\mu_i\in R, t\in N, \sum_{i=1}^t \mu_i=1, z_i\in M, i=\overline{1,t}\}.
\end{equation}
Points $a_1, a_2,...,a_r\in R^d$ are called affine-dependent if one of them lies in the affine hull of the others. Otherwise, $a_1, a_2,...,a_r\in R^d$ is called affine-independent. The maximal affine-independent system of elements of the affine space, $L$, is called the affine basis of $L$. Obviously; the number of elements in the basis does not depend on the choice of basis. One previous small cardinal number from the cardinality of an affine basis is called the affine dimension of  $L$, and is denoted by $dim(L)$. A subset, $Q$, is called convex if it is closed with respect to the convex combinations of its elements, and the empty set is considered as a convex set by definition. The intersection of convex sets is also convex; therefore, the convex hull,  $conv(M)$, of a given nonempty subset, $M\subset R^d$, is defined by the intersection of all convex sets into which it enters. Obviously, $conv(M)=\{\sum_{i=1}^t\mu_i z_i:\mu_i\in R_+, t\in N, \sum_{i=1}^t \mu_i=1, z_i\in M, i=\overline{1,t}\}$. A point of $Q$ which cannot be expressed as a convex combination of other points in $Q$, is called an extreme point of $Q$. The set of all extreme points of $Q$ is denoted by $Extr(Q)$. The convex hull of a finite set is called a polytope. A simplex is defined as the convex hull of affine independent vectors. The following set is called the standard $(d-1)$- dimensional simplex; $$S^{d-1}=\{x =(x_1, x_2,..., x_d)\in R^d:\sum_{i=1}^{d} x_i=1, x_i\geq0, i=\overline{1,d}\}.$$

Throughout the paper we consider the $l_1$ norm on $R^d$  , namely $||x||=|x_1|+|x_2|...+|x_d|$ for $x=(x_1,x_2,...,x_d)\in R^{d}$. Then we can induce a metrics on $Aff(M)$ from this norm, where $M$  is a nonempty set. The interior of $M$  with respect to this induced metric is called the relative interior of $M$  and it is denoted by $ri(M)$. We emphasize that the relative interior of the set does not depend on the choice of norm on $R^d$, since all the norms on $R^d$  are mutually equivalent; therefore, they generate the same topology.

For any $x=(x_1,x_2,...,x_m)\in S^{m-1}$, due to \cite{Mar_Ol}, we define $x_{\downarrow}=(x_{[1]},x_{[2]},...,x_{[m]})$, where $(x_{[1]},x_{[2]},...,x_{[m]})$   is a non-increasing rearrangement of $x$  coordinates i.e. $x_{[1]}\geq x_{[2]}\geq...\geq x_{[m]}$. The point $x_{\downarrow}$  is called a permutation of $x$ in non-increasing. For two elements $x,y$ taken from the simplex, we say that the element $x$ majorizes on $y$ ($y$ majorizes $x$), and write  $x\prec y$ (or $y\succ x$) if the following holds:
$$\sum_{i=1}^{k} x_{[i]}\leq \sum_{i=1}^{k} y_{[i]}$$ for $k=\overline{1,m-1}$.

A geometric illustration can be given in majorization as follows: a vector will be called a permutation vector of $y$  if it is obtained from a permutation of coordinate places of $y$ . Let  $\Pi_{y}$ be the convex hull of all permutation vectors of $y$ . Then according to \cite{Mar_Ol} the following holds:
\begin{prop}\cite{Mar_Ol}
 $y$ majorizes $x$ if and only if $x\in \Pi_{y}$ . In addition, each permutation vector of $y$  is a extreme point of $\Pi_{y}$.
\end{prop}

A continuous operator $V:S^{m-1}\longrightarrow S^{m-1}$  is called a $(m-1)$-dimensional stochastic operator. We call an operator $V:S^{m-1}\longrightarrow S^{m-1}$  a quadratic stochastic operator (QSO) if it has the following form:
$$V(x)_{k}=\sum_{i,j=1}^{m} p_{ij,k} x_{i} x_{j}, $$ for $k=\overline{1,m}$ ,where $x=(x_{1},x_{2},...,x_{m})\in S^{m-1}$, $p_{ij,k}=p_{ji,k}\geq 0$  , $\sum_{r=1}^m p_{ij,r}=1$, for all $i,j,k\in \{1,2,...,m\}=N_{m}$. The quadratic stochastic operator is called the evolution operator in population genetics, and the coefficients $p_{ij,k}$  are called the heredity coefficients of this operator.

It is clear that any QSO is a stochastic operator. By the form of QSO, we infer that any QSO is associated with a unique cubic stochastic matrix of a certain type in the space of real cubic matrices, $M_{m}^{c} (R)$ . By this correspondence, we can define (convex) addition between QSO.
\begin{dfn}
A stochastic operator  $V$ is called bistochastic if $V(x)\prec x$ holds for all $ x\in S^{m-1}$ . A bistochastic QSO is called a quadratic bistochastic operator (QBO). The set of all $m-$dimensional QBOs is denoted by $\mathcal{B}_m$.
\end{dfn}

 By the definition of a majorization, $\mathcal{B}_m$  is a closed set, and it is also closed with respect to the convex sum of its elements; therefore, it is a closed convex subset of $M_{m}^{c} (R)$. An extreme point of $\mathcal{B}_m$  is called extreme QBO, and some necessary conditions and some sufficient conditions for extreme QBO were found in the doctoral thesis of R.Ganikhodjaev \cite{R_G}, but a criteria have not yet been found. According to the definition of majorization, we have that $V((\frac{1}{m},\frac{1}{m},...,\frac{1}{m}))
=(\frac{1}{m},\frac{1}{m},...,\frac{1}{m})$  for a QBO $V$, in other words, the barycenter of the simplex is a fixed point for any bistochastic operator. The following theorem characterizes the main properties of bistochastic operators and $\mathcal{B}_m$.

 \begin{thm}\cite{R_G}
 Let $V:S^{m-1}\rightarrow S^{m-1}$ be a quadratic bistochastic operator then:

 i) $|\omega_{V}(x_{0})|<\infty$, for $\forall x_{0}\in S^{m-1}$, where $\omega_{V}(x_{0})$  ($\omega-$  limit set of $x_0$) is the set of limit points of $\{V^n(x_0)\}_{n=0}^{\infty}$ ;

 ii) $P\circ V$ is quadratic bistochastic operator for any coordinate permutation operator $P:S^{m-1}\rightarrow S^{m-1}$;

 iii) $|Extr(\mathcal{B}_m)|< \infty$.

 \end{thm}
\begin{rem}
A coordinate permutation operator is an operator such that it maps a vector to its permutation vector, and the order of the permutation of the coordinate locations of which does not change when the argument vector changes.
\end{rem}
From the point of view of the theory of dynamical systems, the dynamical system of a certain operator can have very strange behavior. A simpler dynamical system among such strange dynamical systems is that each trajectory in the dynamical system converges at one point. In the theory of QSO, operators having such a simple dynamical system are called regular.
\begin{dfn}
  A QSO is called \textit{regular} if its (forward) trajectories always converge. A regular QSO is called \textit{strictly regular} if it has a unique fixed point.
\end{dfn}

Therefore, a dynamical system of strictly regular QSO is simpler than a dynamical system of regular ones. Some properties of regular QSO were studied in (\cite{First_Ex}-\cite{RFAS}). In particular, the following simple criterion of the regularity of a bistochastic operator is given in \cite{M_S_3} and \cite{RFAS}.
\begin{thm}\label{SSI} (\cite{M_S_3}, \cite{RFAS})
Let  $V:S^{m-1}\to S^{m-1}$ be a bistochastic operator, then $V$ is regular if and only if it has no periodic points except fixed points.
\end{thm}

Obviously, the unique fixed point of the strictly regular bistochastic operator mentioned in the definition is the barycenter of the simplex. Then, we quickly obtain the following simple criterion for strictly regularity of bistochastic operators by Theorem~2.2.

\begin{prop}
A quadratic bistochastic operator is strictly regular if and only if it has no periodic points, except for its unique fixed point $(\frac{1}{m},\frac{1}{m},...,\frac{1}{m})$ .
\end{prop}

\section{Main results}  %% Please avoid complex formulas in (sub)titles

\subsection{On the relative interior of convex polytopes} 
The goal of this subsection is to derive an algebraic expression for points in the relative interior of a convex polytope. The following theorem is suitable for our purpose.
\begin{thm}
Let $f_1,f_2,...,f_s$  be points of $R^d$  such that none lies convex hull of others and $Q=conv\{f_1,f_2,...,f_s\}$. Then $ri(Q)=\{\lambda_1 f_1+\lambda_2 f_2+...+\lambda_s f_s|   \sum_{j=1}^s \lambda_j=1, \lambda_j>0, j=\overline{1;s}\}$.
\end{thm}

We use several lemmas in the proof of the above theorem. Let $A\subset R^d$, $x,y\in R^d$  and $\lambda \in R$, for brevity, we also use the following notations: $A+x:=\{a+x: a\in A\}$, $\lambda A:=\{\lambda a: a\in A\}$,  $[x;y]:=\{\mu x+(1-\mu)y: \mu\in [0;1]\}$. Half-open $[x;y)$, $(x;y]$ and open  $(x;y)$  intervals are defined similarly to the closed interval $[x, y]$.
\begin{lem}\label{l:first}
Let  $Q$  be a nonempty convex subset of $R^d$, then  $ri(Q)\neq\emptyset$.
\end{lem}

\begin{proof}
We consider two cases to prove the lemma.

\textit{Special case:} In this case, we prove the lemma for the simplexes. Let $S\subset R^d$  be a   $d_0-$dimensional simplex (clearly $d_0\leq d$). Then according to the definition of the simplex, there are affine independent vectors $v_1, v_2,...,v_{d_0+1}\in R^d$  such that $S=conv\{v_1,v_2,...,v_{d_0+1}\}$. Consequently, $v_1, v_2,...,v_{d_0+1}$  is an affine basis for $Aff(S)$  due to (\ref{Eq0}), therefore any element $x\in Aff(S)$  can be uniquely expressed as $x=\mu_1 v_1+\mu_2 v_2+...+\mu_{d_0+1} v_{d_0+1}$   with $\mu_1+\mu_2+...+\mu_{d_0}=1$. It follows that the map $\phi:Aff(S)\longrightarrow Aff(S^{d_0})$, defined as $$\phi(\mu_1 v_1+\mu_2 v_2+...+\mu_{d_0+1} v_{d_0+1}):=(\mu_1, \mu_2,...,\mu_{d_0+1})$$ is well-defined. It is clear that $\phi$  is a bijection and a continuous map between  $S$ and the standard  $d_{0}-$dimensional simplex . Obviously,  $\mathcal{G}:=\{(\mu_1,\mu_2,...,\mu_{d_0}):\mu_j>0\}$ is open in $R^{d_0}$, hence $\mathcal{G}_1=\mathcal{G}\bigcap Aff(S^{d_0})$  is open in $Aff(S^{d_0})$. Since $\phi$  is bijection and continuous we see that  $\phi^{-1}(\mathcal{G}_1)=\{\sum_{j=1}^{d_0}\mu_j v_j: (\mu_1,\mu_2,...,\mu_{d_0})\in \mathcal{G}, \sum_{j=1}^{d_0}\mu_j =1\} $ is open set in $Aff(S)$. Now notice that $\phi^{-1}(\mathcal{G}_1)\subset S$, however  $\phi^{-1}(\mathcal{G}_1)$ is an open set. Consequently,   $\phi^{-1}(\mathcal{G}_1)\subset ri(S)$by the definition of relative interior of the set; therefore,  $ri(S)\neq\emptyset$.

\textit{General case:} Let  $Q$ be a nonempty convex subset and $d_0=dim Aff(Q)$. In the case $d_0=0$  there is nothing to prove, so we can assume that $d_0>0$. Then there are  $e_1,e_2,...,e_{d_0+1}\in Q$ such that $\{e_1,e_2,...,e_{d_0+1}\}$ is the affine basis for $Aff(Q)$  by expression (\ref{Eq0}). Let us consider the simplex $S:=conv\{e_1,e_2,...,e_{d_0+1}\}$. Then $S\subset Q$  by the convexity $Q$, so that since  $Aff(S)\subset Aff(Q)$ and $dim Aff(S)=d_0=dim Aff(Q)$, we have $Aff(S)=Aff(Q)$. According to the special case we have $ri(S)\neq \emptyset$ . From this, there exist $\exists x_0\in S$  and open neighborhood  $O_{x_0}$ of $x_0$ such that $O_{x_0}\bigcap Aff(S)\subset S$. Hence according to  $Aff(S)=Aff(Q)$ and  $S\subset Q$ we have $O_{x_0}\bigcap Aff(Q)\subset Q$. It follows from the last inclusion that $x_0$  is a relative interior point of $Q$.
\end{proof}
\begin{lem}\label{l:second}
Let $Q$  be a nonempty closed convex set in $R^d$. Then for any $x\in ri(Q)$  and any $y\in Q/\{x\}$  we have $[x;y)\subset ri(Q)$.
\end{lem}

\begin{proof}
Let  $x_{\lambda}=\lambda x+(1-\lambda)y$ be a point in $(x;y)$. Since $x\in ri(Q)$, it follows that there exists an open neighborhood $\mathcal{O}_x$  of  $x$ such that $\mathbb{O}_x:=\mathcal{O}_x\bigcap Aff(Q)\subset Q$. We have that $\lambda \mathcal{O}_x + (1-\lambda) y$  is open set by the continuity of both addition and multiplying to scalar of vectors. Therefore, $(\lambda \mathcal{O}_x + (1-\lambda) y)\bigcap Aff(Q)= \lambda \mathbb{O}_x + (1-\lambda) y$   is open set in $Aff(Q)$, and hence it is open neighborhood of $x_{\lambda}$  in $Aff(Q)$. $\lambda \mathbb{O}_x + (1-\lambda) y\subset Q$ holds by the convexity of $Q$ , and thus we get  $x_{\lambda}\in ri(Q)$.
\end{proof}

\begin{cor}\label{ImCorr1}
For any nonempty closed convex subset $Q$  of $R^d$  we have $Q=\overline{ri(Q)}$.
\end{cor}

\begin{proof}
$\overline{ri(Q)}\subset Q$ is obvious by the fact that $Q$  is a closed set. Therefore, we should prove that $Q\subset\overline{ri(Q)}$. Let $y\in Q$. There is $\exists x\in ri(Q)$ thanks to Lemma~3.1. We have $[x;y)\subset ri(Q)$  (*) by the Lemma~3.2. We consider an open ball  $\mathcal{O}_y$  centered at $y$, then $\mathcal{O}_y\bigcap [x;y)\neq \emptyset$  and $\mathcal{O}_y\bigcap[x;y)\subset ri(Q)$ (the latter inclusion follows from (*)). From this, we have $y\in \overline{ri(Q)}$.
\end{proof}
\begin{lem}\label{l:third}
For any convex closed set $Q$  in $R^d$  and any point $x\in Q$  the following two conditions are equivalent:

i)	$x\in ri (Q)$;

ii)	For  $\forall y\in Q/\{x\}$, there is $z\in Q$  such that $x\in (y;z)$.

\end{lem}

\begin{proof}
$[i)\Rightarrow ii)]:$ Let $x\in ri(Q)$, then there is  $\exists B_{\delta}(x)$ open ball with centered at $x$  that
\begin{equation}\label{EssIN}
B_{\delta}(x)\bigcap Aff(Q)\subset Q
\end{equation}

We take a point $y$  in $Q/\{x\}$. Since $y\neq x$, we have $||y-x||\neq 0$, and then we can consider the following vector
$$z= -\frac{\delta}{2||y-x||} y+(1+\frac{\delta}{2||y-x||}) x. $$

One the one hand,  $z$ belongs to $Aff(Q)$  as an affine combination of  $x$ and $y$ . On the other hand, $||z-x||=\delta/2<\delta$, thus $z\in B_{\delta}(x)$. Then we have $z\in Q$ by the (\ref{EssIN}). It is easily followed that
$$x=\frac{2||y-x||}{\delta + 2||y-x||} z+\frac{\delta}{\delta + 2||y-x||}y\in (y;z)$$
by the determination of $z$.

$[ii)\Rightarrow i)]:$ We assume that $x\in Q$  is a point which satisfies the condition of the second statement. Since $ri(Q)\neq \emptyset$  (by the Lemma~3.1) we can take a point $y$  in $ri(Q)$. Then there exists  $\exists z\in Q$ that $x\in (y;z)$. According to Lemma3.2  $(y;z)\subset ri(Q)$. Hence we have $x\in ri (Q)$.
\end{proof}

With the above three lemmas at hand we can now pass to the proof of Theorem~3.1.
\begin{proof} (Theorem~3.1)
First, we show that any point  $x$ in $ri(Q)$  can be expressed as a convex combination of $f_1,f_2,...,f_s$  in which the convex combination includes every $f_j$  with a positive coefficient. Since Lemma~3.3 we have $f_j\notin ri(Q)$  for $j=\overline{1,s}$, and thus $x\notin \{f_1,f_2,...,f_s\}=Extr Q$. After this, we successively consider the extreme points $f_1,f_2,...,f_s$. Lemma~3.3 implies the existence of distinct points  $z_1,z_2,...,z_s$ in  $Q$ such that $x\in (f_j;z_j)$  for $j=\overline{1,s}$. Algebraically, this means that there are $\exists \mu_1, \mu_2, ..., \mu_s\in (0;1)$  which
\begin{equation}\label{EssEqs}
x=\mu_j f_j+(1-\mu_j) z_j
\end{equation}
for $j=\overline{1,s}$. Then by these equations we have
\begin{equation} \label{EssEqs2}
x=\frac{1}{s}\sum_{j=1}^s (\mu_j f_j+(1-\mu_j)z_j)=\sum_{j=1}^s \frac{\mu_j}{s} f_j+\sum_{j=1}^s \frac{(1-\mu_j)}{s} z_j.
\end{equation}

Since $Q=conv\{f_1,f_2,...,f_s\}$, each $z_j$  is a convex combination of extreme points $f_1,f_2,...,f_s$. Symbolically, this means that there is a row stochastic matrix $\{\nu_{ji}\}_{j,i=\overline{1,s}}$  such that  $z_j=\sum_{i=1}^s \nu_{ji} f_i$ for all $j$. After replacing each $z_j$ in (\ref{EssEqs2}) with its expression via $f_1,f_2,...,f_s$  we have
\begin{equation}\label{EssEqs3}
x=\sum_{j=1}^s (\frac{\mu_j}{s}+\sum_{k=1}^s \frac {(1-\mu_k)}{s} \nu_{kj}) f_j.
\end{equation}
Clearly, $\frac{\mu_j}{s}+\sum_{k=1}^s \frac {(1-\mu_k)}{s} \nu_{kj}\geq\frac{\mu_j}{s}>0$ and therefore, (\ref{EssEqs3}) is the desired convex combination for $x$.

Now we prove the remaining part of the theorem. Let us assume that $x$  is described as a convex combination of all extreme points as  $x=\lambda_1 f_1+\lambda_2 f_2+...+\lambda_s f_s$ with $\lambda_j>0$  for $j=\overline{1;s}$. Let $y\in Q/\{x\}$, then there are $\exists \sigma_1, ..., \sigma_s\in [0;1)$ with $\sum_{i=1}^s\sigma_i=1$  that $y=\sum_{i=1}^s\sigma_i f_i$. We take $\epsilon:=\frac{min\{\lambda_1,...,\lambda_j\}}{2\cdot max\{\sigma_1,...,\sigma_j\}}>0$  and thus $\lambda_j-\sigma_j\cdot \epsilon\geq \sigma_j\cdot \epsilon>0$  for all $j$. Clearly, $\epsilon<1$ and hence each $\delta_j:=\frac{\lambda_j-\sigma_j\cdot \epsilon}{1-\epsilon}$ is positive and $\delta_1+...+\delta_s=1$. Let us consider $z:=\delta_1 f_1+...+\delta_s f_s$. Then $z\in Q$,  by the its expression via $f_1,...,f_s$, and thus we have
$$(1-\epsilon)z+\epsilon y= \sum_{j=1}^s (\lambda_j-\sigma_j\epsilon) f_j+\epsilon (\sum_{j=1}^s \sigma_j f_j)=\sum_{j=1}^s \lambda_j f_j=x.$$
Hence, by the second statement of Lemma~3.3, we conclude that  $x\in ri(Q)$.

\end{proof}

\subsection{Description of the set of strictly regular quadratic bistochastic operators} Let us consider a convex combination of bistochastic operators $V=\sum_{i=1}^t \lambda_i V_i$, where $\sum_{i=1}^t \lambda_i=1$  and $\lambda_i>1$  for all $i$. Then the following theorem describes the nature of the connection of the dynamical system $V$  with the dynamical systems $V_1,V_2,...,V_t$  and it is the main theorem of the paper.
\begin{thm}
Let $V_1, V_2, ..., V_t$  be   $m$-dimensional bistochastic operators and  \\$\lambda_{1},\lambda_{2},..., \lambda_{t}$ be positive numbers with $\sum_{i=1}^t \lambda_i=1$. Then the following statements hold:

i) $Fix(\sum_{i=1}^t \lambda_{i} V_i)=\bigcap_{i=1}^t Fix(V_i)$

ii)	 $Per_p(\sum_{i=1}^t \lambda_{i} V_i)\subset \bigcap_{i=1}^t Per_p (V_i)$, for all $ p \in N$, where $Fix(V)$  and $Per_p (V)$ are the set of the fixed points and the periodic points of prime period $p$  for  $V$, respectively.

iii) In the general case, the converse of the second statement does not hold.
\end{thm}

%%----------------------------------------------------------------------------
\begin{proof}
Without loss of generality, we can assume that $t = 2$, because the case $t>2$  can easily follows from this case by using mathematical induction principle on $t$.

i) $Fix(\lambda_{1} V_1+\lambda_{2}V_2)\supset Fix(V_1)\bigcap Fix(V_2)$ is obvious, therefore, showing   $Fix(\lambda_{1} V_1+\lambda_{2}V_2)\subset Fix(V_1)\bigcap Fix(V_2)$ is sufficient. Let $x\in Fix(\lambda_{1} V_1+\lambda_{2}V_2)$  that is $\lambda_{1} V_1(x)+\lambda_{2}V_2(x)=x$  (**). Then, we have $V_{1}(x), V_{2}(x)\in \Pi_{x}$ by the bistochasticity of the operators $V_1$ and $V_2$. But $x$  is extreme point of the convex set, $\Pi_x$,  by Proposition~2.1. Then, by the definition of extreme point and (**), we have that $V_{1}(x)=V_{2}(x)=x$.

ii) Let $V=\lambda_1 V_1+\lambda_2 V_2$. We take  $x_{0}\in Per_{p} (V)$ and denote the periodic orbit of $x_0$   by $x_{i}:=V^{i}(x_{0})$ . We also set up the following notations: $y_{i}=V_{1}(x_{i-1})$ and  $z_{i}=V_{2}(x_{i-1})$ for $i=\overline{1,p}$. Since bistochasticity of the operators $V$, $V_1$ and $V_2$,  we have $x_{i},y_{i},z_{i}\in \Pi_{x_{i-1}}$, and $\Pi_{x_{0}}\subset\Pi_{x_{1}}\subset...\subset\Pi_{x_{p}}=\Pi_{x_{0}}$ (due to  $x_p=x_0$) by the definition of these sets. Hence, these sets are equal to each other. Proposition 2.1 implies that each $x_i$ is extreme point of $\Pi_{x_{i}}=\Pi_{x_{0}}$. By the definition of extreme point and according to the equality $\lambda_{1}y_{i}+\lambda_{2}z_{i}=x_{i}\in \Pi_{x_{0}}$ , we get $y_{i}=z_{i}=x_{i}$, for all $i$. Therefore, the three trajectories of  $x_0$  under iteration of  $V$, $V_1$ and $V_2$  are the same and thus each of them is   $p-$periodic. Hence $x_{0}\in Per_{p}(V_{1})\bigcap Per_{p}(V_{2})$.

iii) We take linear bistochastic operators on  $S^2$ that are given by their matrices in the standard basis as:
$$P_1:=\begin{pmatrix} 0&1&0\\
1&0&0\\
0&0&1
\end{pmatrix},
P_2:=\begin{pmatrix} 1&0&0\\
0&0&1\\
0&1&0
\end{pmatrix}.$$
It is easy to verify that these operators are a counterexample to the converse of the second statement.
\end{proof}
\begin{rem} It is worth pointing out that in the proof of the above theorem we do not use the algebraic form of the considered operators. Therefore, in the statement of this theorem, we assert only the bistochasticity of the operators.
\end{rem}
\begin{cor}
Let $V_{1}$  be a strictly regular QBO and $V_{2}$  be a QBO, then $V_\lambda=\lambda V_{1}+(1-\lambda)V_{2}$  is a strictly regular QBO for each point $\lambda$ of $(0;1)$. In particular, the set of strictly regular QBO is convex.
\end{cor}

\begin{proof}
According to Proposition 2.2, we have $Fix(V_{1})=(\frac{1}{m},\frac{1}{m},...,\frac{1}{m})$  and \\ $Per_p(V_1)=\emptyset$ for all $p\geq 2$. Then by Theorem~3.2 we get $Fix(V_{\lambda})=(\frac{1}{m},\frac{1}{m},...,\frac{1}{m})$ and   $Per_p(V_\lambda)=\emptyset$ for each $p\geq 2$. Therefore, we apply again Proposition~2.2 and obtain strictly regularity of  $V_{\lambda}$.
\end{proof}

As mentioned above, the set of $m$-dimensional quadratic bistochastic operators, $\mathcal{B}_m$, is a convex compact set. Therefore, by the Krein-Milman theorem we have
\begin{equation}\label{EssEqs4}
\mathcal{B}_m=\overline{conv(Extr(\mathcal{B}_m))}.
\end{equation}
Similarly, the set of linear bistochastic operators (the Birkhoff polytope) is also a convex compact set, and the Birkhoff-von Neumann theorem states that the extreme points of this set are finite and are coordinate permutation operators. Obviously, the linear bistochastic operator is also a QBO. Now we show that the coordinate permutation operators are also extreme points of the larger set $\mathcal{B}_m$.
\begin{lem}
Let $P$  be a coordinate permutation operator, then $P\in Extr(\mathcal{B}_m)$.
\end{lem}

\begin{proof}
First, we show that the identity operator is an extreme QBO. The bistochasticity of the identity operator is obvious and let  $\lambda V_{1}+(1-\lambda) V_{2}= id$ for some $\lambda \in (0;1)$, and  $V_{1}, V_{2}\in \mathcal{B}_m$. Then $\lambda V_{1}(x)+(1-\lambda) V_{2}(x)=x$ for all $x \in S^{m-1}$, and it follows that $V_{1}(x), V_{2}(x)\in \Pi_{x}$  by the bistochasticity of the operators $V_1$ and $V_2$. Hence, we have that $V_{1}(x)=V_{2}(x)=x$  by Proposition 2.1 and definition of extreme point. By taking $x$ arbitrarily in $S^{m-1}$, we get $V_{1}=V_{2}=id$, so that the identity operator is an extreme QBO.

Let $P$ be a coordinate permutation operator, therefore, it is invertible and its inverse is also coordinate permutation operator, and hence both of them are QBO. Let  $\lambda V_{1}+(1-\lambda) V_{2}= P$ for some $\lambda \in (0;1)$, and $V_{1}, V_{2}\in \mathcal{B}_m$. From this, we have $\lambda (P^{-1}\circ V_{1})+(1-\lambda)(P^{-1}\circ V_{2})= id$  by linearity of  $P$ (both $P^{-1}\circ V_{1}$  and $P^{-1}\circ V_{2}$  are QBO according to the second statement of Theorem~2.1). We have shown above that the identity operator is an extreme QBO, and thus we get $P^{-1}\circ V_{1}=P^{-1}\circ V_{2}=id$. Hence  $V_{1}=V_{2}=P$.
\end{proof}

We denote the group of   $m-$dimensional coordinate permutation operators by $\mathcal{P}_m$. Clearly,  $|\mathcal{P}_m|=m!$. We let $\{P_j\}_{j=\overline{1,m!}}$  denote the elements in $\mathcal{P}_m$. Lemma~3.4 asserts that $\mathcal{P}_m\subset Extr(\mathcal{B}_m)$  and the third statement of Theorem 2.1 states that the extreme points of $\mathcal{B}_m$  are finite. Let  $s=|Extr(\mathcal{B}_m)\setminus\mathcal{P}_m|$, and $\{\mathcal{V}_1, \mathcal{V}_2,...,\mathcal{V}_s\}=Extr(\mathcal{B}_m)\setminus\mathcal{P}_m$. Consequently, $Extr(\mathcal{B}_m)=\{\mathcal{V}_1, \mathcal{V}_2,...,\mathcal{V}_s\}\cup\mathcal{P}_m$ and thus, by the Krein-Milman theorem, we get $$\mathcal{B}_m=conv(P_1,..., P_{m!},\mathcal{V}_1, \mathcal{V}_2,...,\mathcal{V}_s).$$

\begin{thm}
Any operator in $ri (\mathcal{B}_m)$ is strictly regular.
\end{thm}

\begin{proof}
Let  $V\in ri(\mathcal{B}_m)$ be an operator in the relatively interior of  $\mathcal{B}_m$, then it can be expressed as
$$V= \sum_{j=1}^{s} \lambda_{j}\mathcal{V}_{j}+\sum_{j=1}^{m!}\lambda_{j+s} P_j$$
by Theorem~3.1, where $\sum_{j=1}^{s+m!}\lambda_j=1$  and $\lambda_j>0$  for $j=\overline{1,(s+m!)}$. Then the first statement of Theorem~3.2 implies
\begin{equation}\label{Eq1}
Fix(V)=\bigcap_{j=1}^s Fix(\mathcal{V}_j)\bigcap\bigcap_{j=1}^{m!} Fix(P_j)\subset \bigcap_{j=1}^{m!} Fix(P_j)=(\frac{1}{m},\frac{1}{m},...,\frac{1}{m})
\end{equation}
and according to the second statement of Theorem~3.2, we get
\begin{equation}\label{In1}
Per_{p}(V)\subset\bigcap_{j=1}^s Per_p(\mathcal{V}_j)\bigcap\bigcap_{j=1}^{m!} Per_p(P_j)\subset\bigcap_{j=1}^{m!} Per_p(P_j)=\emptyset.
\end{equation}
for all natural number $p\geq 2$.
Hence we obtain strictly regularity of $V$  by Proposition~2.2
\end{proof}

\begin{rem}
We note that Theorem~3.3 is proved using geometrical principles, and the proof is based on Theorem~3.2. This theorem can also be followed by the main theorem of \cite{Second_Ex}  (Theorem 3.1 in that work).

\end{rem}

\begin{cor}
The set of strictly regular QBO is dense in $\mathcal{B}_m$.
\end{cor}

\begin{proof}
According to Corollary~\ref{ImCorr1}, we have $\overline{ri(\mathcal{B}_m)}=\mathcal{B}_m$.
\end{proof}

\section{Examples}

In this section we give examples of the strictly regular QBO in any dimensions with the following theorem.
\begin{thm}
Let $A=\{a_{ij}\}_{i,j=1}^m$  be a bistochastic matrix, then the following statements hold:

i) $V_{A}(x)_k=\sum_{i=1}^m {a_{ki}x_i^2+x_k(1-x_k)}$  ($k=\overline{1,m}$) is a bistochastic operator;

ii)	$V(x)_{k}={1\over m}\sum_{i=1}^m x_i^2+x_k(1-x_k)$   ($k=\overline{1,m}$) is a strictly regular QBO.

\end{thm}
%------------------------------------------------------------------------------
\begin{proof}
i) Let  $x'=V(x)$. Then $x'\prec x$, and by the definition of majorization, $\sum_{s=1}^k x_{i_s}'\leq \sum_{j=1}^k x_{[j]}$ for all $\forall k\in N_{m} =\{1,2,...,m\}$  and  $\{i_1,i_2,...,i_k\}\subset N_m$. Let  $x_{\downarrow}=(x_{\pi(1)},x_{\pi(2)}),...,x_{\pi(m)})$, namely  $\pi\in S_m$ is suitable to the permutation of the coordinates of $x$  in non-increasing order, where $S_m$  is the permutation group of  $N_m$. Firstly, we will show that $\sum_{j=1}^m(\sum_{s=1}^k a_{i_s j})x_j^2\leq\sum_{j=1}^k x_{[j]}^2$  (*) for each multi-index $\{i_1,i_2,...,i_k\}\subset N_m$. Indeed, by the bistochasticity of the matrix  $A$, it is followed that  $\sum_{s=1}^k a_{i_s t}\leq 1$ for all $t\in N_m$, and  $\sum_{j=1}^m(\sum_{s=1}^k a_{i_s j})=\sum_{j=1}^m(\sum_{s=1}^k a_{i_s \pi(j)})=k$. Therefore,
\begin{equation}
\sum_{j=k+1}^m \sum_{s=1}^k a_{i_s\pi(j)}=k-\sum_{j=1}^k\sum_{s=1}^k a_{i_s\pi(j)}=\sum_{j=1}^k(1-\sum_{s=1}^k a_{i_s\pi(j)}).
\end{equation}
Thus, by $x_{\pi(m)}\leq x_{\pi(m-1)}\leq...\leq x_{\pi(1)}$, we have
\begin{equation}
\sum_{j=k+1}^m (\sum_{s=1}^k a_{i_s\pi(j)})x_{\pi(j)}^2\leq \sum_{j=1}^k(1-\sum_{s=1}^k a_{i_s\pi(j)})x_{\pi(j)}^2.
\end{equation}
From this, it follows that $\sum_{j=1}^m(\sum_{s=1}^k a_{i_sj})x_j^2=\sum_{j=1}^m(\sum_{s=1}^k a_{i_s\pi(j)})x_{\pi(j)}^2\leq\sum_{j=1}^k x_{\pi(j)}^2=\sum_{j=1}^k x_{[j]}^2$.

We denote by $x_{[i_1]},x_{[i_2]},...,x_{[i_k]}$  the non-increasing rearrangement of  $x_{i_1},x_{i_2},...,x_{i_k}$, then by the definition of $V_A$  and $x\in S^{m-1}$, we have
$\sum_{s=1}^k x_{i_s}'=\sum_{j=1}^m(\sum_{s=1}^k a_{i_sj})x_j^2+\sum_{s=1}^k x_{i_s}(1-x_{i_s})\leq\sum_{s=1}^k x_{[s]}^2+\sum_{s=1}^k x_{[i_s]}(1-x_{[i_s]})=\sum_{s=1}^k(x_{[s]}+(x_{[s]}-x_{[i_s]})
(x_{[s]}+x_{[i_s]}-1))\leq\sum_{s=1}^k x_{[s]}$ (the last inequality follows from $(x_{[s]}- x_{[i_s]})(x_{[s]}+x_{[i_s]}-1)\leq 0$).

ii) Applying the previous part of the theorem to the matrix $A=\{\frac{1}{m}\}_{i,j=\overline{1;m}}$ (i.e. all elements of the matrix equal to each other), we obtain that $V$  is a bistochastic operator.
Let  $C_{\sigma}:=\{x\in S^{m-1}: x_{\sigma(1)}\geq x_{\sigma(2)}\geq...\geq x_{\sigma(m)})\}$, where $\sigma\in S_m$. We take an element $x$ in $C_{\sigma}$, then
$V(x)_{\sigma(i)}-V(x)_{\sigma(j)}=\frac{1}{m}\sum_{i=1}^m x_i^2+x_{\sigma(i)}(1-x_{\sigma(i)})-\frac{1}{m}\sum_{i=1}^m x_i^2-x_{\sigma(j)}(1-x_{\sigma(j)})=x_{\sigma(i)}(1-x_{\sigma(i)})
-x_{\sigma(j)}(1-x_{\sigma(j)})=(x_{\sigma(i)}-x_{\sigma(j)})(1-x_{\sigma(i)})-x_{\sigma(j)})\geq 0$
for each $i,j\in N_m$ with $i<j$. Hence, we get $V(x)\in C_{\sigma}$, i.e. each $C_{\sigma}$  is invariant under $C_{\sigma}$. Therefore, any trajectory $V$  converges to some point in $Fix(V)$  by the bistochasticity of the operator $V$. Let $p\in Fix(V)$, then $V(p)=p$  implies that $p_i^2={1\over m}\sum_{i=1}^m p_i^2$  for all $i$. Hence, we get $p=(\frac{1}{m},\frac{1}{m},...,\frac{1}{m})$, and thus $Fix(V)=\{(\frac{1}{m},\frac{1}{m},...,\frac{1}{m})\}$. Since $Fix(V)=\{(\frac{1}{m},\frac{1}{m},...,\frac{1}{m})\}$, it follows that any trajectory $V$  converges to the unique fixed point $(\frac{1}{m},\frac{1}{m},...,\frac{1}{m})$. Consequently, by the definition, $V$   is a strictly regular bistochastic operator.
\end{proof}

\begin{rem}
We note that in the above theorem, the strictly regularity of the operator is proved by using the fact that it is order-preserving map (monotone).  It is worth mentioned that the second statement of the theorem can also be proved applying the main theorem of \cite{Second_Ex} (Theorem~3.1 in that paper) and this method of proof is completely different from ours.
\end{rem}

\textbf{Acknowledgments.} I would like to express deep gratitude to professors U.Rozikov, R.Ghanikhodjaev and U.Jamilov for many useful discussions, M.Saburov for attentive reading of the text and for making many useful comments.

\bibliographystyle{amsplain}

\end{document}